\documentclass[11pt]{amsart}

\usepackage{amsmath}

\usepackage{amssymb}

\usepackage{amsfonts}

\usepackage{amsthm}

\usepackage{enumerate}

 \usepackage{esint}
 
 \usepackage{tikz}
 
 \usetikzlibrary{patterns}
 
 \usepackage{pgfplots}
 
  \usepackage{tikz}
 
 \usetikzlibrary{patterns}
 
 \usepackage{pgfplots}

 \usepackage{subcaption}
 \usepackage{graphicx}
 \usepackage{lipsum}

\usepackage{graphicx}

\usepackage{verbatim}

 \usepackage{subcaption}
 \usepackage{graphicx}
 \usepackage{lipsum}

\usepackage{graphicx}

\usepackage{verbatim}

\newcommand{\wt}{\widetilde}

\newcommand{\Norm}[1]{ \left\|  #1 \right\| }

\newcommand{\floor}[1]{\lfloor #1 \rfloor }

\newcommand{\Be}{\begin{equation}}

\newcommand{\Ee}{\end{equation}}

\newcommand{\Bm}{\begin{multline}}

\newcommand{\Em}{\end{multline}}

\newcommand{\Bea}{\begin{eqnarray}}

\newcommand{\Eea}{\end{eqnarray}}

\newcommand{\Beas}{\begin{eqnarray*}}

\newcommand{\Eeas}{\end{eqnarray*}}

\newcommand{\Benu}{\begin{enumerate}}

\newcommand{\Eenu}{\end{enumerate}}

\newcommand{\Bi}{\begin{itemize}}

\newcommand{\Ei}{\end{itemize}}

\def\intslash{\fint}

\def\intslash{\rlap{\kern  .32em $\mspace {.5mu}\backslash$ }\int}

\def\qsl{{\rlap{\kern  .32em $\mspace {.5mu}\backslash$ }\int_{Q_x}}}

\def\Norm#1{{ \left\|  #1 \right\| }}

\def\floor#1{{\lfloor #1 \rfloor }}

\def\emph#1{{\it #1 }}

\def\supp{{\text{\rm supp}}}

\def\cM{{\mathcal {M}}}

 at 10 true pt

\def\Log{{\text{Log}}}

\def\be#1{\begin{equation}\label{ #1}}

\def\endeq{\end{equation}}

\def\endal{\end{align}}

\def\bas{\begin{align*}}

\def\eas{\end{align*}}

\def\bi{\begin{itemize}}

\def\ei{\end{itemize}}

\def\emph#1{{\it #1}}

\def\textbf#1{{\bf #1}}

\theoremstyle{plain}

   \newtheorem{theorem}{Theorem}[section]

   \newtheorem{proposition}[theorem]{Proposition}

   \newtheorem{lemma}[theorem]{Lemma}

   \newtheorem{theorem*}{Theorem}

\theoremstyle{remark}

   \newtheorem{remark}[theorem]{Remark}

\theoremstyle{definition}

   \newtheorem{definition}[theorem]{Definition}

\numberwithin{equation}{section}

\begin{document}
\title[The lacunary spherical maximal operator]{Improved endpoint bounds for the lacunary spherical maximal operator}
\author[L. Cladek]{Laura Cladek}
\author[B. Krause]{Benjamin Krause}

\address{L. Cladek, Department of Mathematics\\ The University of British Columbia\\1984 Mathematics Road, Vancouver, BC V6T 1Z2}
\email{cladek@math.ubc.ca}

\address{B. Krause, Department of Mathematics\\ The University of British Columbia\\1984 Mathematics Road, Vancouver, BC V6T 1Z2}
\email{benkrause@math.ubc.ca}


\maketitle 

\begin{abstract}
We prove new endpoint bounds for the lacunary spherical maximal operator and as a consequence obtain almost everywhere pointwise convergence of lacunary spherical means for functions locally in $L\log\log\log L(\log\log\log\log L)^{1+\epsilon}$ for any $\epsilon>0$.
 \end{abstract}

\section{Introduction}
For $k\in\mathbb{Z}$ define the spherical means of radius $2^k$ on functions $f$ on $\mathbb{R}^d$ by
\begin{align*}
\mathcal{A}_kf(x)=f\ast\sigma_k,
\end{align*}
where $\sigma_k$ denotes the ($L^1$-normalized) surface measure on the $(d-1)$-sphere of radius $2^k$ centered at the origin.  Define the lacunary spherical maximal operator $\mathcal{M}$ by 
\begin{align*}
\mathcal{M}f(x)=\sup_{k\in\mathbb{Z}}|\mathcal{A}_kf(x)|.
\end{align*}
For convience, we inductively define the notation
\begin{align*}
\Log(t):=\log_2(100+t),
\end{align*}
\begin{align*}
\Log^1(t):=\Log(t),
\end{align*}
\begin{align*}
\Log^n(t):=\Log(100+{\Log}^{n-1}(t)), \qquad n\in\mathbb{N}\setminus\{0, 1\}.
\end{align*}
It was shown by C. Calder\'{o}n (\cite{cal}) and Coifman and Weiss (\cite{cw}) that $\mathcal{M}$ extends to a bounded operator on $L^p(\mathbb{R}^d)$ for $p>1$, which implies almost everywhere pointwise convergence of lacunary spherical means for functions in $L^p(\mathbb{R}^d)$ for $p>1$. A new proof of this result was later given by Duoandikoetxea and Rubio de Francia in \cite{dr}. It has remained open, however, as to whether $\mathcal{M}$ is weak type $(1, 1)$, or equivalently, whether almost everywhere pointwise convergence of lacunary spherical means holds for functions in $L^1(\mathbb{R}^d)$. 
\newline
\indent
In \cite{christstein}, Christ and Stein showed using an extrapolation argument that $\mathcal{M}f\in L^{1, \infty}(\mathbb{R}^d)$ for functions $f$ on $\mathbb{R}^d$ supported in a cube $Q$ satisfying $f\in L\log L(Q)$. Christ also proved in \cite{christ1} that $\mathcal{M}$ maps the Hardy space $H^1(\mathbb{R}^d)$ to $L^{1, \infty}(\mathbb{R}^d)$. More recently, Seeger, Tao, and Wright showed in \cite{stw}, \cite{stw2} that $\mathcal{M}$ maps the space $L\,\Log^2 L$ to $L^{1, \infty}(\mathbb{R}^d)$. In this paper we prove that $\mathcal{M}$ maps all characteristic functions in $L\,\Log^3 L$ to $L^{1, \infty}(\mathbb{R}^d)$, and more generally maps the entire space $L\, \Log^3 L (\Log^4 L)^{1+\epsilon}$ to $L^{1, \infty}(\mathbb{R}^d)$ for every $\epsilon>0$, thus obtaining almost everywhere pointwise convergence of lacunary spherical means for functions locally in $L\,\Log^3 L (\Log^4 L)^{1+\epsilon}$. 

\begin{proposition}\label{mainpropo}
There exists a constant $C$ such that for all measurable $f=\chi_E$ and all $\alpha>0$ we have
\begin{align}\label{mainineq}
|\{x\in\mathbb{R}^d:\,\mathcal{M}f(x)>\alpha\}|\le C\alpha^{-1}\int|f(x)|{\Log}^3(|f(x)|/\alpha)\,dx.
\end{align}
\end{proposition}

\begin{proposition}\label{mainpropo2}
For every $\epsilon>0$, there exists a constant $C(\epsilon)$ such that for all measurable $f$ and all $\alpha>0$ we have
\begin{multline}\label{mainineq}
|\{x\in\mathbb{R}^d:\,\mathcal{M}f(x)>\alpha\}|
\\
\le C(\epsilon)\alpha^{-1}\int|f(x)|\Log^3(|f(x)|/\alpha)(\Log^4(|f(x)|/\alpha))^{1+\epsilon}\,dx.
\end{multline}
\end{proposition}

We obtain as a corollary the following theorem.
\begin{theorem}
Let $\epsilon>0$, and let $f$ be locally in $L\,\Log^3 L\,(\Log^4 L)^{1+\epsilon}$. Then 
\begin{align*}
\mathcal{A}_kf(x)\to f(x)
\end{align*}
for almost every $x\in\mathbb{R}^d$.
\end{theorem}

\indent
Before we proceed with the proofs, we briefly outline the argument. In \cite{stw}, the restricted version of the argument relied crucially on a decomposition of the function $f=\chi_E$ on Whitney cubes into characteristic functions of sets called ``generalized boxes", which had properties called ``length" and ``thickness." As the name suggests, in two dimensions such sets are a generalization of rectangular boxes, for which the length and thickness correspond to the long and short sides respectively of the rectangle. In the case of two dimensions, convolution of a rectangular box with the measure $\sigma_k$ has measure equal to $2^k$ times the length of the box. Similarly, the length of a generalized box determines for how many scales $k$ one may throw away the support of $\sigma_k$ convolved with the characteristic function of the generalized box. Conversely, the thickness of the box determines what $L^2$ estimates one may obtain for $\sigma_k$ convolved with the characteristic function of the generalized box. 
\newline
\indent
The argument of \cite{stw} proceeded by combining standard Calderon-Zygmund techniques along with this decomposition of $E$ into generalized boxes on Whitney cubes, and by leveraging $L^2$ and exceptional set size estimates via the properties of length and thickness for each generalized box. Our argument will also make use of a similar decomposition, but there will also be many new ingredients involved, and in general our argument will more closely use the geometry of the sphere.
\newline
\indent 
For example, we exploit the hyperrectangular cap structure of the spherical measures to introduce a fairly involved algorithm for defining exceptional sets, and we throw away more exceptional sets than in \cite{stw}. These exceptional sets are defined by covering $\mathbb{R}^d$ with collections of rotated hyperrectangular grids, where the dimensions of the hyperrectangles are determined by an iterative relationship between the dimensions of a given generalized box and the cap structure of the spherical measure. On fixing a particular direction in $S^{d-1}$ which determines the orientation of the hyperrectangular grids to be considered, we then subdivide the generalized box into hyperrectangular pieces where the generalized box has sufficiently high ``mass", and throw away as an exceptional set the sumset of this hyperrectangular box and a piece of the cap on the sphere with normals pointing in similar directions as the short side of the hyperrectangular box, so that such a set is is contained in a translation of the fattening of the spherical cap by an amount comparable to the short side of the hyperrectangular box. 
\newline
\indent
We then decompose the kernel $\sigma_k\ast\sigma_k$ into linear combinations of characteristic functions of hyperrectangles with dimensions corresponding to the caps that appear in our algorithm for defining exceptional sets. The $L^2$ estimates for each such piece of the kernel convolved with a given hyperrectangular piece in a grid with similar orientation is determined by the mass of the generalized box on that hyperrectangular piece. There are essentially double-logarithmically (in the relevant parameter) many such different sizes of caps that appear, which alone would lead to the desired $L^2$ estimates with an additional double-logarithmic factor. However, we are able to throw away triple-logarithmically many ``intermediate scales", that is, we may sum in $L^1$ the convolutions of characteristic functions of parts of the generalized boxes with intermediate masses with the associated cap measures. After doing so, we improve the $L^2$ estimates for the remaining ``light scales" by the needed double-logarithmic factor, and also improve the support size estimates for the remaining ``heavy scales" by a double-logarithmic factor.

\section{Preliminary Reductions}
Most of this paper will be devoted to the proof of Proposition \ref{mainpropo}. We will in fact prove the following slightly stronger version of Proposition \ref{mainpropo}, which will be useful in the proof of Proposition \ref{mainpropo2}.

\begin{proposition}\label{midprop}
There exists a constant $C$ such that for all $\alpha>0$ and all measurable $f$ such that $\Log^3(|f(x)|/\alpha)$ is essentially constant on the support of $f$ we have
\begin{align*}
|\{x\in\mathbb{R}^d:\,\mathcal{M}f(x)>\alpha\}|\le C\alpha^{-1}\int|f(x)|\Log^3(|f(x)|/\alpha)\,dx.
\end{align*}
\end{proposition}

At the end of the paper we will prove Proposition \ref{mainpropo2}. This will follow easily from the proof of Proposition \ref{midprop}, as the proof uses only $L^1$ and $L^2$ methods.
\newline
\indent
To prove Proposition \ref{midprop}, we may assume without loss of generality that $f\le 1$ and that $\alpha<1$, and further that $\alpha\le f\le 1$ on the support of $f$. By scaling considerations, it suffices to prove Proposition \ref{midprop} in the special case that $\Log^3(|f(x)|/\alpha)\approx\Log^3(\alpha^{-1})$ on the support of $f$. For any number $\delta>0$, let $\mathcal{Q}_{\delta}$ be a grid of cubes of sidelength $\delta$ partitioning $\mathbb{R}^d$. By taking limits, we may assume that there is some small number $\delta>0$ such that $f$ is a linear combination of characteristic functions of cubes $Q\in\mathcal{Q}_{\delta}$, i.e. that $f$ is \textit{granular} in the sense defined in \cite{stw}. We refer to $\delta$ as the \textit{grain size} of $f$. Also, throughout what follows, all logarithms will be assumed to be base $2$. 

\subsection*{Reductions using Calder\'{o}n-Zygmund Theory}
Similarly to \cite{stw}, we first make some standard reductions using Calder\'{o}n-Zygmund theory. Let $M_{HL}$ denote the Hardy-Littlewood maximal operator, and let
\begin{align*}
\Omega=\{x: M_{HL}(f)(x)>\alpha\}.
\end{align*}
Note that $f\chi_{\Omega^c}\le 2^{10}\alpha$, and by the $L^2$-boundedness of the lacunary spherical maximal operator and Chebyshev's inequality we have
\begin{align*}
|\{x: \mathcal{M}(f\chi_{\Omega^c})>\alpha\}|&\lesssim\alpha^{-2}\Norm{\cM(f\chi_{\Omega^c})}_2^2\lesssim\alpha^{-2}\Norm{f\chi_{\Omega^c}}_2^2
\\\lesssim\alpha^{-1}\Norm{f}_1.
\end{align*}
It thus suffices to prove the desired inequality with $f$ replaced by $f\chi_{\Omega}$. Let $\mathfrak{O}$ be the set of Whitney cubes obtained by applying a Whitney decomposition to $\Omega$. We define $f_q:=f\chi_{q}$ for any $q\in\mathfrak{O}$.
\newline
\indent
We now introduce some cancellation, by defining the projection operator $\Pi_q$ to be the projection operator onto a certain space of polynomials. This is a standard Calder\'{o}n-Zygmund theory technique also used in \cite{stw}.  That is, let $\{P_j\}_{j=1}^L$ be an orthonormal basis for the space of polynomials of degree $\le 100d$ on the unit cube $[-1/2, 1/2]^d$. If $q$ is a cube with center $x_q$ and sidelength $l(q)$, define
\begin{align*}
|\Pi_q[h](x)|=\chi_q(x)\sum_{j=1}^LP_j(\frac{x-x_q}{l(q)})\int_qh(y)P_j(\frac{y-x_q}{l(q)})\,\frac{dy}{l(q)^d}.
\end{align*}
Write
\begin{align*}
b_q=f_q-\Pi_q[f_q].
\end{align*}
Since 
\begin{align*}
|\Pi_q[f_q](x)|\lesssim\alpha,
\end{align*} 
it follows by the $L^2$-boundedness of the lacunary spherical maximal operator that
\begin{align*}
\Norm{\sup_k|\sum_q\Pi_q[f_q]\ast\sigma_k|}_2^2\lesssim\alpha\Norm{f}_1,
\end{align*}
and so by Chebyshev's inequality it remains to prove that
\begin{align*}
|\{x:\,\sup_k|\sum_qb_q\ast\sigma_k|>\alpha\}|\lesssim\alpha^{-1}\Log^3(\alpha^{-1})\Norm{f}_1.
\end{align*}
We will replace the sup by an $\ell^2$ norm and in fact show
\begin{align*}
|\{x:\,\bigg(\sum_k\big|\sum_qb_q\ast\sigma_k\big|^2\bigg)^{1/2}>\alpha\}|\lesssim\alpha^{-1}\Log^3(\alpha^{-1})\Norm{f}_1.
\end{align*}
We further reduce this to showing
\begin{align}\label{fir}
|\{x:\,\bigg(\sum_q\sum_k\big|b_q\ast\sigma_k\big|^2\bigg)^{1/2}>\alpha\}|\lesssim\alpha^{-1}\Log^3(\alpha^{-1})\Norm{f}_1
\end{align}
and
\begin{align}\label{sec}
|\{x:\,\bigg(\sum_k\sum_{q\ne q^{\prime}}(b_q\ast\sigma_k)(\overline{b_{q^{\prime}}\ast\sigma_k})\bigg)^{1/2}>\alpha\}|\lesssim\alpha^{-1}\Log^3(\alpha^{-1})\Norm{f}_1
\end{align}
We will first eliminate the easier of these two inequalities, (\ref{sec}). By Chebyshev, this reduces to proving
\begin{multline*}
\Norm{\bigg(\sum_k\sum_{q\ne q^{\prime}}(b_q\ast\sigma_k)(\overline{b_{q^{\prime}}\ast\sigma_k})\bigg)^{1/2}}_2^2\lesssim
\\
\sum_k\sum_{q\ne q^{\prime}}\bigg|\left<b_q\ast\sigma_k, b_{q^{\prime}}\ast\sigma_k\right>\bigg|\lesssim\alpha\Norm{f}_1,
\end{multline*}
which can be proven as in \cite{stw} by exploiting the smoothness of the kernel $\sigma_k\ast\sigma_k$ and the cancellation of $b_q$. Since the proof is nearly identical to that given in \cite{stw}, we omit it here. 
\newline
\indent
It remains to prove (\ref{fir}). By throwing away the unions of the supports of $\sigma_k\ast\chi_{q}$ for $k$ such that $2^k\le l(q)$, which has total measure $\lesssim \sum_q|q|\lesssim\alpha^{-1}|E|$, we may restrict the sum in $k$ to range over $k$ such that $2^k>l(q)$. It is easy to see that
\begin{multline*}
\sum_q\sum_{k:\, 2^k>l(q)}\Norm{\Pi_q[f_q]\ast\sigma_k}_2^2\lesssim\sum_q\sum_{k:\,2^k>l(q)}\Norm{\alpha\chi_q\ast\sigma_k}_2^2
\\
\lesssim\alpha\sum_q\Norm{f_q}_1\lesssim\alpha\Norm{f}_1.
\end{multline*}
Thus we may replace the left hand side of (\ref{fir}) by the same expression with each occurrence of $b_q$ replaced by $f_q$. We have thus reduced Proposition \ref{mainpropo} to proving the following proposition.
\begin{proposition}\label{mainprop}
Let $f\le 1$ and let $\alpha<1$. Let $f_q$ be defined as above, i.e. $f_q=f\chi_q$ where $q\in\mathfrak{Q}$, the collection of Whitney cubes for $f$ at height $\alpha$. Then we have the inequality
\begin{align}\label{mainineq}
|\{x:\,\bigg(\sum_{q\in\mathfrak{Q}}\sum_{k:\,2^k>l(q)}\big|\sigma_k\ast f_q\big|^2\bigg)^{1/2}>\alpha\}|\lesssim\alpha^{-1}\Log^3(\alpha^{-1})\Norm{f}_1.
\end{align}
\end{proposition}

\subsection*{Structural decomposition of $f_q$}
In \cite{stw}, the support of $f_q$ was decomposed into structures referred to as ``generalized boxes", which behaved in a certain way like $1$-dimensional sets and which had associated quantities referred to as ``length" and ``thickness," the former which governed support size estimates and the latter which controlled $L^2$ bounds. We describe a decomposition of $f_q$ that is in a similar spirit. List all dyadic numbers $\gamma$ between $\alpha^2 l(q)^{d-1}$ and $l(q)^{d-1}$ in increasing order as $\gamma_0, \gamma_1,\ldots, \gamma_N$. We decompose $E_q={\text{supp}}(f_q)$ into structures $E_q^{\gamma_i}$ similar to generalized boxes. we do this by first excising the dyadic cubes on which $f_q$ is ``heaviest", i.e., of ``weight" $\gamma_N$, and collect those cubes as $E_q^{\gamma_N}$. Next, we continue to $\gamma_{N-1}$ and excise all dyadic cubes on which the remainder of $f_q$ is of weight $\gamma_{N-1}$, and collect those cubes as $E_q^{\gamma_{N-1}}$, and then continue this process. We thus inductively define
\begin{align*}
{E_q^{\gamma_N}}:=\bigcup_{Q\text{ dyadic},\,l(Q)\le 2^{10}l(q):\,\int_Q|f_q|\ge \gamma_N\cdot l(Q)}Q,
\end{align*}
and for $1<j<N$,
\begin{align*}
E_q^{\gamma_j}:=\bigcup_{Q\text{ dyadic}, \,l(Q)\le 2^{10}l(q):\,  \int_Q|f_q-f_q\chi_{\bigcup_{l>j}E_q^{\gamma_l}}|\ge\gamma_j\cdot l(Q)}Q.
\end{align*}
Set
\begin{align*}
E_q^{\gamma_0}:=q\setminus\bigg(\bigcup_{1\le j\le N}E_q^{\gamma_j}\bigg).
\end{align*}
Inductively define
\begin{align*}
f_q^{\gamma_N}=f_q\chi_{E_q^{\gamma_N}},
\end{align*}
and for $0\le j<N$, define
\begin{align*}
f_q^{\gamma_j}=(f_q-\sum_{l>j}f_q^{\gamma_l})\chi_{E_q^{\gamma_j}}.
\end{align*}
It follows easily by induction on $j$ that for every $j$ we have
\begin{align}\label{ind1}
\sum_{l>j}f_q^{\gamma_l}=f_q\chi_{\bigcup_{l>j}E_q^{\gamma_l}},
\end{align}
so that we indeed have a decomposition
\begin{align}\label{decomp}
f_q=\sum_{j=1}^Nf_q^{\gamma_j}.
\end{align}
We will refer to $f_q^{\gamma}$ as having \textit{critical density} $\gamma$. Let $\mathcal{Q}$ be a maximal cover of $E_q^{\gamma}$ by dyadic cubes (so that $E_q^{\gamma}=\bigcup_{Q\in\mathcal{Q}}Q$), and define the \textit{length} $\lambda(f_q^{\gamma})$ of $f_q^{\gamma}$ as
\begin{align}\label{lengthdef}
\lambda(f_q^{\gamma}):=\sum_{Q\in\mathcal{Q}}l(Q).
\end{align}
Then by construction of the $f_q^{\gamma_j}$ and (\ref{ind1}), we have
\begin{align}\label{width1}
\int|f_q^{\gamma_j}|\approx\gamma_j\cdot\lambda(f_q^{\gamma_j}), \qquad j>0
\end{align}
and
\begin{align}\label{width2}
\int|f_q^{\gamma_0}|\lesssim\gamma_0\cdot\lambda(f_q^{\gamma_0}).
\end{align}
Moreover, for every cube $Q$,
\begin{align}\label{width3}
\int_Q|f_q^{\gamma}|\lesssim\gamma\cdot l(Q).
\end{align}

We will use the decomposition (\ref{decomp}) in an essential way throughout the rest of the paper, as well as the key properties (\ref{width1}), (\ref{width2}) and (\ref{width3}). In \cite{stw}, the analog of (\ref{width2}) is that for every generalized box $B$ of thickness $\gamma$ and length $\lambda$, we have $|B|\approx\gamma\cdot\lambda$. For convenience, we restate this decomposition as the following lemma.

\begin{lemma}[Structural decomposition lemma]\label{decomplemma}
Let $f\le 1$ and let $\alpha<1$, and for each $q\in\mathfrak{Q}$ define $f_q:=f\chi_q$, where $\mathfrak{Q}$ is the collection of Whitney cubes for $f$ at height $\alpha$. List the dyadic numbers between $\alpha^2l(q)^{d-1}$ and $l(q)^{d-1}$ in increasing order as $\gamma_0, \gamma_1, \ldots, \gamma_N$. Then we can decompose
\begin{align}\label{decomp2}
f_q=\sum_{j=1}^Nf_q^{\gamma_j},
\end{align}
so that if for each $j$ we define the \textit{length} $\lambda(f_q^{\gamma_j})$ of $f_q^{\gamma_j}$ as in (\ref{lengthdef}), then we have that (\ref{width1}), (\ref{width2}) and (\ref{width3}) holds. We refer to $\gamma_j$ as the \textit{critical density} of $f_q^{\gamma_j}$.
\end{lemma}

\subsection*{Further reductions}
First, we show that we may get rid of the term $\sigma_k\ast{f_q^{\gamma_0}}$, so that we may replace the left hand side of (\ref{mainineq}) in Proposition \ref{mainprop} with
\begin{align*}
|\{x:\bigg(\sum_q\sum_{k:\, 2^k>l(q)}|\sum_{\gamma>\gamma_0}\sigma_k\ast f_q^{\gamma}|^2\bigg)^{1/2}>\alpha\}|.
\end{align*}
Indeed, we have the well-known pointwise estimate
\begin{align}\label{point}
\sigma_k\ast\sigma_k(x)\lesssim 2^{-k(d-1)}|x|^{-1}\chi_{|x|\le 2^{k+1}}.
\end{align}
For $2^k\ge l(q)$ we take $L^2$ norms and use (\ref{point}) to obtain the estimate
\begin{multline*}
\Norm{\sigma_k\ast f_q^{\gamma_0}}_2^2=\left<f_q^{\gamma_0}, \sigma_k\ast\sigma_k\ast f_q^{\gamma_0}\right>
\\
\lesssim 2^{-k(d-1)}\int|f_q^{\gamma_0}(x)|\bigg(\sum_{l: \alpha^{2/(d-1)}l(q)\le 2^l\le 100l(q)}\int_{y:\,|x-y|\approx 2^l}\frac{1}{|x-y|}|f_q^{\gamma_0}(y)|\,dy
\\
+\sum_{l: 2^l<\alpha^{2/(d-1)}l(q)}\int_{y:\,|x-y|\approx 2^l}\frac{1}{|x-y|}|f_q^{\gamma_0}(y)|\,dy
\bigg).
\end{multline*}
Now we use (\ref{width3}) to deal with the first term and the fact that $f_q^{\gamma_0}\le f\le 1$ to deal with the second term, which yields
\begin{multline}\label{2ess}
\Norm{\sigma_k\ast f_q^{\gamma_0}}_2^2\lesssim 2^{-k(d-1)}\int|f_q^{\gamma_0}|\bigg(\sum_{l: \alpha^{2/(d-1)}l(q)\le 2^l\le 100l(q)}\alpha^2l(q)^{d-1}
\\
+\sum_{l: 2^l<\alpha^{2/(d-1)}l(q)}2^{l(d-1)}\bigg)
\\
\lesssim2^{-k(d-1)}\Norm{f_q^{\gamma_0}}_1(\alpha^2l(q)^{d-1}\log(\alpha^{-1})\lesssim\alpha\Norm{f_q^{\gamma_0}}_12^{-k(d-1)}l(q)^{d-1}.
\end{multline}
\newline
\indent
Next, for $k(d-1)<\log(\alpha^{-1}\gamma)$, we may throw away the support of $\sigma_k\ast f_q^{\gamma}$, since
\begin{align*}
|\text{supp}(\sigma_k\ast f_q^{\gamma})|\lesssim 2^{k(d-1)}\lambda(f_q^{\gamma}),
\end{align*}
and hence by (\ref{width1}),
\begin{multline}\label{suppess}
\bigg|\bigcup_{\gamma, k:\, k(d-1)<\log(\alpha^{-1}\gamma)}\text{supp}(\sigma_k\ast f_q^{\gamma})\bigg|\lesssim\sum_{\gamma}\alpha^{-1}\gamma\lambda(f_q^{\gamma})
\\
\lesssim\sum_{\gamma}\alpha^{-1}\Norm{f_q^{\gamma}}_1\lesssim\alpha^{-1}\Norm{f_q}_1.
\end{multline}
This will allow us to discard sufficiently small scales $2^k$. We also will discard sufficiently large scales $2^k$ by proving as in \cite{stw} that we have the estimate
\begin{align}\label{2ess2}
\Norm{\sigma_k\ast{f_q^{\gamma}}}_2^2\lesssim \gamma\log(\alpha^{-1})2^{-k(d-1)}\Norm{f_q^{\gamma}}_1.
\end{align}
Indeed, the pointwise estimate
\begin{align*}
\sigma_k\ast\sigma_k(x)\lesssim 2^{-k(d-1)}|x|^{-1}\chi_{|x|\le 2^{k+1}}
\end{align*}
implies that
\begin{multline*}
\Norm{\sigma_k\ast f_q^{\gamma}}_2^2=\left<{f_q^{\gamma}}, \sigma_k\ast\sigma_k\ast{f_q}^{\gamma}\right>
\\
\lesssim2^{-k(d-1)}\bigg(\int |f_q^{\gamma}(x)|\sum_{\log(\gamma)/(d-1)\le j\le 100\log(l(q))}\int_{|x-y|\approx 2^j}\frac{1}{|x-y|}|f_q^{\gamma}(y)|\,dy\,dx
\\
+\int|f_q^{\gamma}(x)|\sum_{j\le\log(\gamma)/(d-1)}\int_{|x-y|\approx 2^j}\frac{1}{|x-y|}|f_q^{\gamma}(y)|\,dy\,dx\bigg)
\end{multline*}
We may put the integral in $x$ in each term inside the sum in $j$, and then we split the integral in $x$ over a sum of dyadic cubes $\mathcal{Q}_j$ of sidelength $2^j$. Using (\ref{width1}) to deal with the first term and the fact that $f_q^{\gamma}\le f_q\le 1$ for the second term yields
\begin{multline*}
\Norm{\sigma_k\ast f_q^{\gamma}}_2^2\lesssim2^{-k(d-1)}\bigg(\sum_{\log(\gamma)/(d-1)\le j\le 100\log(l(q))}2^{-j}\sum_{Q\in\mathcal{Q}_j}2^j\gamma \int_Q|f_q^{\gamma}|
\\
+\sum_{j\le\log(\gamma)/(d-1)}\sum_{Q\in\mathcal{Q}_j}2^{j(d-1)}\int_Q|f_q^{\gamma}|\bigg)
\\
\lesssim 2^{-k(d-1)}\gamma\log(\alpha^{-1})\int|f_q^{\gamma}|,
\end{multline*}
since $\gamma\gtrsim \alpha^2l(q)^{d-1}$, and so (\ref{2ess2}) is proved. It follows that 
\begin{align*}
\sum_{k:\,k(d-1)\ge\log(\gamma\alpha^{-1})+100\log(\log(\alpha^{-1}))}\Norm{\sigma_k\ast{f_q^{\gamma}}}_2^2\lesssim\alpha(\log(\alpha^{-1}))^{-10}\Norm{f_q^{\gamma}}_1,
\end{align*}
and hence by Cauchy-Schwarz,
\begin{multline}\label{2ess3}
\sum_k\Norm{\sigma_k\ast\sum_{\gamma:\,k(d-1)\ge\log(\gamma\alpha^{-1})+100\log(\log(\alpha^{-1}))}f_q^{\gamma}}_2^2
\\
\lesssim\sum_k(\log(\alpha^{-1}))\sum_{\gamma:\,k(d-1)\ge\log(\gamma\alpha^{-1})+100\log(\log(\alpha^{-1}))}\Norm{\sigma_k\ast{f_q^{\gamma}}}_2^2
\\
\lesssim\sum_{\gamma}\alpha\Norm{f_q^{\gamma}}_1\lesssim\alpha\Norm{f_q}_1.
\end{multline}
Combining (\ref{2ess}), (\ref{suppess}) and (\ref{2ess3}) we see that we have reduced Proposition \ref{mainprop} to the following.
\begin{proposition}\label{mainprop2}
Let $f\le 1$ and let $\alpha<1$. Let $f_q=f\chi_q$ for each $q\in\mathfrak{Q}$, where $\mathfrak{Q}$ is the collection of Whitney cubes for $f$ at height $\alpha$. Decompose
\begin{align*}
f_q=\sum_{j=0}^Nf_q^{\gamma_j}.
\end{align*}
as in Lemma \ref{decomplemma}. We have the inequality
\begin{multline}\label{mainineq3}
|\{x:\,\bigg(\sum_q\sum_{k:\,2^k>l(q)}\big|\sum_{\gamma:\, \log(\alpha^{-1}\gamma)\le k(d-1)\le \log(\alpha^{-1}\gamma)+100\Log^2(\alpha^{-1})}\sigma_k\ast{f_q^{\gamma}}\big|^2\bigg)^{1/2}
\\
>\alpha\}|\lesssim\alpha^{-1}Log^3(\alpha^{-1})\Norm{f}_1.
\end{multline}
\end{proposition}

\section{Proof of Proposition \ref{mainprop2}}

Our plan for proving Proposition \ref{mainprop2} is loosely as follows. For $q\in\mathfrak{O}$, we will show that for each $k$ we may write 
\begin{align*}
\sigma_k\ast{f_q}=g_{1, k, q}+g_{2, k, q}
\end{align*} 
for nonnegative functions $g_{1, k, q}$ and $g_{2, k, q}$ so that there is a set $A_{k}$ such that $A:=\bigcup_k A_{k}:=\bigcup_{k, q}A_{q, k}$ has measure $\lesssim \alpha^{-1}\Norm{f}_1$ and so that
\begin{align}\label{mainineq2}
\sum_q\sum_{k>l(q)}\Norm{g_{1, k, q}}_{L^2(\mathbb{R}^2\setminus A_{k})}^2\lesssim\alpha\Norm{f}_1
\end{align} 
and 
\begin{align*}
\Norm{\sum_q\sum_{k>l(q)} g_{2, k, q}}_1\lesssim \Norm{f}_1\log(\log(\log(\alpha^{-1}))).
\end{align*}

\subsection*{Throwing away exceptional sets}

We will now introduce an algorithm for defining exceptional sets. Suppose we have fixed a function $f\le 1$, an $\alpha<1$, and suppose also that we have fixed a $q\in\mathfrak{Q}$, where $\mathfrak{Q}$ is the collection of Whitney cubes for $f$ at height $\alpha$. Let $f_q=\sum_{j=0}^Nf_q^{\gamma_j}$ be the decomposition of $f_q$ from Lemma \ref{decomplemma}. Fix a critical density $\gamma=\gamma_j$ for some $j$. Our algorithm will depend on a fixed parameter $M>0$ which can be viewed as the ``height" of the exceptional sets thrown away. The larger the value of $M$, the ``heavier" that $f_q^{\gamma}$ will be on the exceptional sets thrown away. 
\subsubsection*{Defining double-logarithmically many scales}
The scales that we define in this subsection are motivated by a decomposition of the kernel $\sigma_k\ast\sigma_k$ into linear combinations of characteristic functions of hyperrectangles, which will appear later in the paper. We first fix some $k$ such that 
\begin{align*}
2^k\ge\max(l(q), (\gamma\cdot\alpha^{-1})^{1/(d-1)}),
\end{align*} 
since these are the relevant values of $k$ in Proposition \ref{mainprop2}. Thus we assume that the parameters $k, \gamma, M$ are fixed. We begin by identifying $O(\Log^2(\alpha^{-1}))$ many natural scales in our problem that will lead us to our $L\,\Log^3 L$ result. The scales may be enumerated in increasing order as $\{c_i\}_{i=0}^N$, where 
\begin{align}\label{scale}
c_j=\max(\gamma^{2^{-j}/(d-1)}l(q)^{1-2^{-j}}, (\gamma^{1/(d-1)}\alpha^{-1})^{(1-2^{-j}/(d-1))}).
\end{align}
\indent
We now describe how these scales arise; they in fact arise from a relationship between the parameters $\gamma$ and $\alpha$ and the geometry of the sphere of radius $2^k$ that we have fixed.
\newline
\indent
The first scale would be the diameter of a spherical cap of thickness $\gamma^{1/(d-1)}$ on a sphere of radius $\max(l(q), (\gamma\cdot\alpha^{-1})^{1/(d-1)})$, which is 
\begin{align*}
\approx \max(\gamma^{1/2(d-1)}l(q)^{1/2}, \gamma^{1/(d-1)}\cdot\alpha^{-1/2(d-1)}).
\end{align*} 
The next scale would be the diameter of a spherical cap of thickness $$\max(\gamma^{1/2(d-1)}l(q)^{1/2}, \gamma^{1/(d-1)}\cdot\alpha^{-1/2(d-1)}))$$ on the sphere of radius $\max(l(q), (\gamma\cdot\alpha^{-1})^{1/(d-1)})$. Continuing in this manner, letting each scale be the diameter of a spherical cap on the sphere of radius $\max(l(q), (\gamma\cdot\alpha^{-1})^{1/(d-1)})$ whose thickness is the previous scale, the $j^{\text{th}}$ scale would be $\max(\gamma^{2^{-j}/(d-1)}l(q)^{1-2^{-j}}, (\gamma^{1/(d-1)}\alpha^{-1})^{(1-2^{-j}/(d-1))})$. We stop when 
\begin{multline*}
c_j=\max(\gamma^{2^{-j}/(d-1)}l(q)^{1-2^{-j}}, (\gamma^{1/(d-1)}\alpha^{-1})^{(1-2^{-j}/(d-1))})
\\
\ge 2^{-10}\max(l(q), (\gamma\cdot\alpha^{-1})^{1/(d-1)}).
\end{multline*} 
Since $\gamma>\alpha^2l(q)^{d-1}$, this will happen before $j=10\lceil{\log(\log(\alpha^{-1}))}\rceil$. We enumerate our scales as $\{c_i\}_{i=0}^N$
in increasing order, where each $c_j$ is given by (\ref{scale}).

\begin{figure}
\begin{tikzpicture}[scale=0.7]

\draw[thick] (0,0) circle (4cm);
\draw[blue] (-0.4, 3.92) rectangle (.4, 4.12) node[anchor=east] {\small $c_0$\,\,\,\,\,\,\,\,\,\,\,\,} node[anchor=south] {\small $c_1$\,\,\,\,\,\,\,\,\,\,\,\,}
node[anchor=west] {\,\,\,\,\,\small$c_1$} ; 

\draw[blue, rotate=-35] (-1, 3.6) rectangle (0.8, 4.4) node[anchor=north east ] {\small $c_2$\,\,\,\,\,\,\,\,\,\,\,\,};
\draw[blue, rotate=-90] (-2, 3) rectangle (2, 4.8) node[anchor=north] {\small $c_2$\,\,\,\,\,\,\,\,\,\,\,\,\,\,\,\,\,\,\,\,} ;
\filldraw[blue] (4, 0) circle (0pt) node[anchor=west] {\small \,\,\,\,\,\,\,\,\,\,$c_3$};

\end{tikzpicture}
\caption{The scales $c_j$. Here $c_0=\gamma$ and the circle has radius $2^k$ for some $k$ with $2^k\ge \max(l(q), (\gamma\cdot\alpha^{-1})^{1/(d-1)})$.}\label{fig1}
\end{figure}
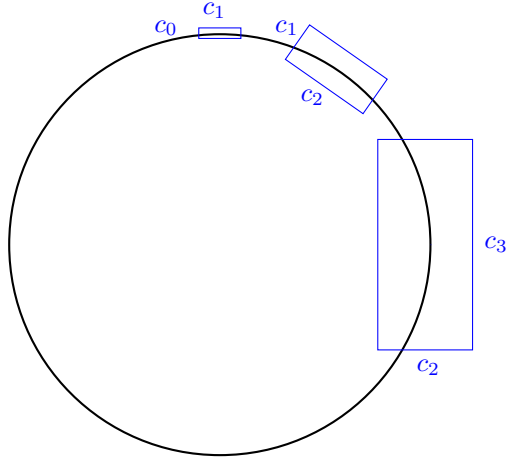

\subsubsection*{Throwing away exceptional sets at each scale}
First, we give some definitions and notation that lead up to our definition of the exceptional sets. We assume in what follows that we are in the situation described in the subsequent remark.
\begin{remark}\label{remarkassume}
In what follows, we assume that we have fixed a function $f\le 1$, an $\alpha<1$, and also fixed $q\in\mathfrak{Q}$, where $\mathfrak{Q}$ is the collection of Whitney cubes for $f$ at height $\alpha$. Assume we have decomposed $f_q=\sum_{j=0}^Nf_q^{\gamma_j}$ as in Lemma \ref{decomplemma}, and fixed some critical density $\gamma=\gamma_j$ for some $j>0$. Finally, assume we have fixed some $k$ such that
\begin{align*}
2^k\ge \max(l(q), (\gamma\cdot\alpha^{-1})^{1/(d-1)}).
\end{align*}
Lastly, assume that we have also fixed some parameter $M>0$ which we will refer to as a ``height parameter" for the collection of sets we will subsequently define. Define the scales $\{c_j\}_{j=0}^N$ as in (\ref{scale}), with $N$ chosen to be the smallest number so that 
\begin{align*}
c_N\ge 2^{-10}\max(l(q), (\gamma\cdot\alpha^{-1})^{1/(d-1)}).
\end{align*}
(As noted earlier, this implies $N\le 15\log(\log(\alpha^{-1}))$.)
\end{remark}
\begin{definition}\label{defdir}
Assume we are in the situation of Remark \ref{remarkassume} and have already fixed the parameters mentioned in that remark. For $1\le i\le N$, define $\{\Phi_{j, i}\}_j$ to be some set of $\approx (c_i/c_{i-1})^{d-1}$ essentially  equally spaced directions in $S^{d-1}$. 
\end{definition} 
\begin{definition}\label{defrect}
Assume we are in the situation of Remark \ref{remarkassume} and have already fixed the parameters mentioned in that remark. For a fixed $i$ with $1\le i\le N$, consider the collection $\{R\}$ of all $c_i\times c_i\times\cdots\times c_i\times c_{i-1}$ hyperrectangles that belong to some fixed grid (say, the grid centered at the origin) parallel and orthogonal to the direction $\Phi_{j, i}$. Define $\mathcal{R}_{j, i}$ to be the collection of all such rectangles $R$ such that
\begin{align}\label{hea}
\int_R|f_q^{\gamma}|\ge c_{i-1}M\gamma.
\end{align}
\end{definition} 
\begin{definition}\label{defsets}
Assume we are in the situation of Remark \ref{remarkassume} and have already fixed the parameters mentioned in that lemma. For $1\le i\le N$, define $\sigma_{k, j, i}$ to be the surface measure of a spherical cap on the sphere of radius $2^k$ of angular width $100c_{i-1}/c_i$ with some normal in direction $\Phi_{j, i}$, normalized so that $\Norm{\sigma_{k, j, i}}_1\lesssim c_{i-1}/c_i$. Define the set $A_{j, i}$ as
\begin{align*}
A_{j, i}:=\text{supp}\bigg(\sigma_{k, j, i}\ast\chi_{\bigcup_{R\in\mathcal{R}_{j, i}}}100R\bigg).
\end{align*}
\end{definition}

\begin{definition}\label{defexset}
Assume we are in the situation of Remark \ref{remarkassume} and have already fixed the parameters mentioned in that lemma. Set
\begin{align*}
S_{M, k, q, \gamma, i}:=\bigcup_jA_{j, i}
\end{align*}
We define our exceptional set as
\begin{align*}
S_{M, k, q, \gamma}:=\bigcup_{i=1}^NS_{M, k, q, \gamma, i}.
\end{align*}
\end{definition}

We claim the following upper bound on the size of the exceptional sets defined above.

\begin{lemma}\label{sizelemma}
In light of Definitions \ref{defdir}, \ref{defrect}, \ref{defsets} and \ref{defexset}, we have the bound
\begin{align*}
|S_{M, k, q, \gamma}|\lesssim M^{-1}2^{k(d-1)}\cdot\lambda(f_q^{\gamma}).
\end{align*}
\end{lemma}

Before we proceed with the proof of Lemma \ref{sizelemma}, we say a few words to motivate the previous definitions. We note that in some sense the support of $\sigma_{k, j, i}$ is ``adapted" to translates of hyperrectangles in $\mathcal{R}_{j, i}$, in the sense that convolution with characteristic functions of hyperrectangles effectively fattens it by $c_{i-1}$ and translates it. Thus we note that for each $R\in\mathcal{R}_{j, i}$, the set $\text{supp}(\sigma_{k, j, i}\ast\chi_{100R})$ is contained in a $1000c_{i-1}$-neighborhood of a translate of the cap $\text{supp}(\sigma_{k, j, i})$. The hyperrectangles $R\in\mathcal{R}_{j, i}$ that are sufficiently ``heavy" in the sense of (\ref{hea}) correspond to (more or less) poor $L^2$ estimates for $\sigma_{k, j, i}\ast\chi_R$, and so we would like to remove the supports of $\sigma_{k, j, i}\ast\chi_R$. Since the support of this is essentially contained in a $c_{i-1}$-fattening of the cap $\text{supp}(\sigma_{k, j, i})$, the heavier the rectangles we consider (the larger $M$ is) the fewer number of such rectangles there can be, so the smaller the total size of exceptional sets thrown away. Thus using a pigeonholing argument, we can obtain the bound from Lemma \ref{sizelemma}.

\begin{proof}[Proof of Lemma \ref{sizelemma}]
For $1\le n\le N$, consider $S_{M, k, q, \gamma, n}$; as noted above, $S_{M, k, q, \gamma, n}$ is contained in a union of $1000c_{n-1}$-neighborhoods of translates of caps on the sphere of radius $2^k$ of angular width $100c_{n}/c_{n-1}$, each having a normal vector in some direction in $\{\Phi_{j, n}\}_j$.
\begin{definition}
For any $n, n^{\prime}$ with $n>n^{\prime}$, we will say that $\Phi_{j, n}$ is a \textit{parent} of $\Phi_{j^{\prime}, n^{\prime}}$ and write $(j^{\prime}, n^{\prime})\lesssim (j, n)$ if $|\Phi_{j, n}-\Phi_{j', n^{\prime}}|\le |\Phi_{j'', n}-\Phi_{j', n'}|$ for any $j''$, i.e. if $\Phi_{j, n}$ is a closest vector in the $n^{\text{th}}$ generation of vectors $\{\Phi_{j, n}\}$ to $\Phi_{j^{\prime}, n^{\prime}}$.
\end{definition} 
By definition, $S_{M, k, q, \gamma}$ is contained in a set of the form
\begin{align*}
\bigcup_{1\le n\le N}\bigcup_j\bigcup_{R\in\mathcal{R}_{j, n}}\text{supp}(\sigma_{k, j, n}\ast\chi_{1000R}),
\end{align*}
where each $R\in\mathcal{R}_{j, n}$ is a $c_n\times c_n\times\cdots\times c_n\times c_{n-1}$ hyperrectangle with short side pointing in direction $\Phi_{j, n}$ satisfying 
\begin{align}\label{rec}
\int_R|f_{q}^{\gamma}|\ge c_{n-1}M\gamma.
\end{align}
Moreover, if $(j', n')\lesssim (j, n)$ then for any $R\in\mathcal{R}_{j, n}$ and any $R'\in\mathcal{R}_{j', n'}$, if $R\cap R'\ne\emptyset$ then $R'\subset 100R$. It follows that we can choose subcollections $\wt{\mathcal{R}}_{j, n}\subset\mathcal{R}_{j, n}$ satisfying
\begin{align*}
S_{M, k, q, \gamma}\subset \bigcup_{1\le n\le N}\bigcup_j\bigcup_{R\in\wt{\mathcal{R}}_{j, n}}\text{supp}(\sigma_{k, j, n}\ast\chi_{10000R}),
\end{align*}
so that for given any direction $\Phi_{j_1, 1}$, a chain of parents
\begin{align}\label{par}
\Phi_{j_1, 1}\lesssim\Phi_{j_2, 2}\lesssim\ldots\lesssim\Phi_{j_N, N}
\end{align}
satisfies the property that the hyperrectangles in the collections $\wt{\mathcal{R}}_{j_n, n}$ are all pairwise disjoint.
\newline
\indent
Since $\sigma_{k, j, n}\ast\chi_{10000R}$ is essentially supported in a $c_{n-1}$-fattening of a the cap $\text{supp}(\sigma_{k, j, n}\ast\chi_{10000R})$, the measure of its support is $\lesssim 2^{k(d-1)}c_{n-1}^d/c_n^{d-1}$. We can thus bound
\begin{multline*}
|S_{M, k, q, \gamma}|\lesssim\bigg|\bigcup_{1\le n\le N}\bigcup_j\bigcup_{R\in\wt{\mathcal{R}}_{j, n}}\text{supp}(\sigma_{k, j, n}\ast\chi_{10000R})\bigg|
\\
\lesssim\sum_n\sum_j2^{k(d-1)}\frac{c_{n-1}^d}{c_n^{d-1}}\cdot \text{card}(\wt{\mathcal{R}}_{j, n}).
\end{multline*}
By the disjointness property mentioned above and (\ref{rec}), for a chain of parents as in (\ref{par}), we have the bound
\begin{align}\label{upbound}
\int|f_q^{\gamma}|\ge \sum_{1\le n\le N}\sum_{R\in\wt{\mathcal{R}}_{j_n, n}}\int_R|f_q^{\gamma}|\ge \sum_{1\le n\le N}c_{n-1}M\gamma\cdot\text{card}(\wt{\mathcal{R}}_{j_n, n}).
\end{align}
Since each direction $\Phi_{j, n}$ is the parent of $\lesssim(\frac{c_{n-1}}{c_n}\cdot\frac{c_1}{c_0})^{d-1}$ many directions $\Phi_{j', n'}$ with $n'=1$, it follows that
\begin{multline*}
|S_{M, k, q, \gamma}|\lesssim\sum_n\sum_j2^{k(d-1)}\frac{c_{n-1}^d}{c_n^{d-1}}\cdot \text{card}(\wt{\mathcal{R}}_{j, n})
\\
\lesssim 2^{k(d-1)}\sum_{j_1}\sum_{j, n:\, (j, n)\gtrsim (j_1, 1)}2^{k(d-1)}\frac{c_{n-1}^d}{c_n^{d-1}}\cdot (\frac{c_{n-1}}{c_n}\cdot\frac{c_1}{c_0})^{1-d}\cdot\text{card}(\wt{\mathcal{R}}_{j, n})
\\
\lesssim 2^{k(d-1)}\sum_{j_1}\sum_{j, n:\, (j, n)\gtrsim (j_1, 1)}c_{n-1}(\frac{c_0}{c_1})^{d-1}\cdot\text{card}(\wt{\mathcal{R}}_{j, n})
\\
\lesssim 2^{k(d-1)}(\frac{c_0}{c_1})^{d-1}\sum_{j_1}\int|f_q^{\gamma}|M^{-1}\gamma^{-1}\lesssim 2^{k(d-1)}\int|f_q^{\gamma}|M^{-1}\gamma^{-1}
\\
\lesssim M^{-1}2^{k(d-1)}\cdot\lambda(f_q^{\gamma}),
\end{multline*}
where in the penultimate line we have used (\ref{upbound}) and the fact that there are $O((\frac{c_1}{c_0})^{d-1})$ many values of $j_1$ to sum over, and in the last line we have used (\ref{width1}).
\end{proof}

\subsection*{Choice of the height parameter for the exceptional sets}
Suppose we are in the situation of Remark \ref{remarkassume}, with all relevant parameters fixed as in the remark except for $k$. First we define three different regimes in which $k$ can live. Let $\mathcal{K}_1\subset\mathbb{Z}$ be the set of all $k$ satisfying
\begin{align*}
 0\le k(d-1)<\max(\log(l(q))(d-1), \log(\gamma\alpha^{-1})).
 \end{align*}
Let $\mathcal{K}_2\subset\mathbb{Z}$ be the set of all $k$ satisfying
\begin{multline*}
\max(\log(l(q))(d-1), \log(\gamma\alpha^{-1}))\le k(d-1)
\\
\le\log(\gamma\alpha^{-1})+100\log(\log(\alpha^{-1})).
\end{multline*}
Let $\mathcal{K}_3\subset\mathbb{Z}$ be the set of all $k\notin\mathcal{K}_2$ satisfying
\begin{align*}
k(d-1)>\max(\log(l(q))(d-1), \log(\gamma\alpha^{-1})+100\log(\log(\alpha^{-1}))),
\end{align*}
so that we may write $\mathbb{N}=\mathcal{K}_1\sqcup\mathcal{K}_2\sqcup\mathcal{K}_3$. For each $k$, we define the set $A_{q, k, \gamma}$ by

\[A_{q, k, \gamma}:= \begin{cases} 
      \text{supp}(\sigma_k\ast f_q^{\gamma}) & k\in\mathcal{K}_1 \\
      S_{2^{k(d-1)}\alpha\gamma^{-1}\log(\log(\alpha^{-1})), k, q, \gamma} & k\in\mathcal{K}_2 \\
      \emptyset & k\in\mathcal{K}_3. 
   \end{cases}
\]

\begin{lemma}
We have the bound
\begin{align*}
\bigcup_{k, q, \gamma}A_{k, q, \gamma}\lesssim\alpha^{-1}\Norm{f_q}_1.
\end{align*}
\end{lemma}

\begin{proof}
For $k\in\mathcal{K}_1$, with $k(d-1)\le\log(l(q))$, we have
\begin{align*}
|A_{q, k, \gamma}|\lesssim 2^{k(d-1)}l(q),
\end{align*}
and summing over $k$ gives the desired bound.
For $k\in\mathcal{K}_1$, with $k(d-1)\le\log(\gamma\alpha^{-1})$, we have
\begin{align*}
|A_{q, k, \gamma}|\lesssim 2^{k(d-1)}\lambda(f_q^{\gamma}).
\end{align*}
and summing over $k$ gives the desired bound, since for the largest such $k$ we have
\begin{align*}
2^{k(d-1)}\lambda(f_q^{\gamma})\lesssim\alpha^{-1}\gamma\lambda(f_q^{\gamma})\lesssim\alpha^{-1}\Norm{f_q^{\gamma}}_1
\end{align*}
by (\ref{width1}). Thus we have shown
\begin{align*}
\bigcup_{k\in\mathcal{K}_1, q, \gamma}|A_{k, q, \gamma}|\lesssim\alpha^{-1}\Norm{f_q}_1.
\end{align*}
It remains to consider $\mathcal{K}_2$. By Lemma \ref{sizelemma} and (\ref{width1}), we have
\begin{align*}
|S_{2^{k(d-1)}\alpha\gamma^{-1}\log(\log(\alpha^{-1})), k, q, \gamma}|\lesssim \alpha^{-1}(\Log^2(\alpha^{-1}))^{-1}\gamma\cdot\lambda(f_q^{\gamma})
\\ \lesssim\alpha^{-1}(\Log^2(\alpha^{-1}))^{-1}\Norm{f_q^{\gamma}}_1.
\end{align*}
Summing over $\gamma$ and $k\in\mathcal{K}_2$ and using the fact that $|\mathcal{K}_2|\lesssim\Log^2(\alpha^{-1})$ gives the desired bound.
\end{proof}
After throwing away the exceptional set $\bigcup_{k, q, \gamma}A_{k, q, \gamma}$, we have reduced Proposition \ref{mainprop2} to proving the following proposition.
\begin{proposition}\label{mainprop3}
Let $f\le 1$ and let $\alpha<1$. Let $f_q=f\chi_q$ for each $q\in\mathfrak{Q}$, where $\mathfrak{Q}$ is the collection of Whitney cubes for $f$ at height $\alpha$. Decompose
\begin{align*}
f_q=\sum_{j=0}^Nf_q^{\gamma_j}.
\end{align*}
as in Lemma \ref{decomplemma}. We have the inequality
\begin{multline}\label{mainineq3}
|\{x:\,\bigg(\sum_q\sum_{k:\,2^k>l(q)}\big|\sum_{\gamma:\, \log(\alpha^{-1}\gamma)\le k(d-1)\le \log(\alpha^{-1}\gamma)+100\Log^2(\alpha^{-1})}\sigma_k\ast{f_q^{\gamma}}\big|^2\bigg)^{1/2}
\\
>\alpha\}|\cap \bigg(\bigcup_{q, k, \gamma}A_{q, k, \gamma}\bigg)^c\lesssim\alpha^{-1}Log^3(\alpha^{-1})\Norm{f}_1.
\end{multline}
\end{proposition}

We proceed with the proof of Proposition \ref{mainprop2} by a combination of $L^1$ and $L^2$ techniques. First, we determine to which parts of the functions $\sigma_k\ast f_q^{\gamma}$ we will applying $L^1$ techniques and to which we will apply $L^2$ techniques.

\subsection*{A decomposition of $\sigma_k\ast{f_q^{\gamma}}$}
Again, in what follows assume that we are in the setup of Remark \ref{remarkassume}, except that we will not fix the parameter $M$ and instead let it vary. Set $i=\floor{\log(M)}$. Recall that $f$ is granular with grain size $\delta$ for some small number $\delta>0$, i.e. that $f=\sum_lc_l\chi_{\omega_l}$ where each $\omega_l$ is a $\delta$-grain, i.e. a cube of sidelength $\delta$. We now associate a natural spherical measure to each $\delta$-grain $\omega_l$, defined so that it is supported on those caps where there exists a ``heavy" rectangle containing $\omega_l$ with short side essentially pointing in the direction normal to the corresponding cap.
\begin{definition}
For each $\delta$-grain $\omega_l$ and for a given $i$, define $\sigma_{k, \omega_l}^{i, \gamma}$ to be the surface measure on
\begin{align*}
\bigcup_{(n, j):\,\exists R\in\mathcal{R}_{j, n, n-1}(\omega_l\cap R\ne\emptyset)}\text{supp}(\sigma_{k, j, n}),
\end{align*}
where $n$ ranges over $1\le n\le N$, normalized so that the total measure is $2^{k(d-1)}\Theta$, where $\Theta$ is the total angle subtended by the spherical caps in the support of $\sigma_{k, \omega_l}^{i, \gamma}$.
\end{definition}
Recall that the parameter $i$ corresponds to the ``height", in a sense, of the spherical measure $\sigma_{k, \omega_l, i, \gamma}$. We now decompose the function $\sigma_k\ast f_q^{\gamma}$ into different ``heights", and identify a critical height below which we are able to obtain efficient $L^2$ estimates.
\begin{definition}
We identify the critical height $2^{m(k, q, \gamma)}$ by defining
\begin{align*}
m(k, q, \gamma):=k(d-1)-\log(\gamma\alpha^{-1})-\lceil{100\log(\log(\log(\alpha^{-1})))}\rceil.
\end{align*} 
For a given height $2^i$, define the ``projection of $\sigma_k\ast f_q^{\gamma}$ onto height $2^i$" as
\begin{align*}
g_k^{i, \gamma}=\sum_{\delta-\text{cubes }\omega_l}\sigma_{k, \omega_l}^{i, \gamma}\ast{(f_q^{\gamma}\chi_{\omega_l})}.
\end{align*}
\end{definition}
With these definitions in mind, we have the decomposition
\begin{align}\label{compdec}
\sigma_k\ast {f_q^{\gamma}}=\sigma_k\ast{f_q^{\gamma}}-g_k^{m(k, q, \gamma), \gamma}+\sum_{i\ge m(k, q, \gamma)}(g_k^{i, \gamma}-g_k^{i+1, \gamma}).
\end{align}
As previously mentioned, we will see that we have efficient (even when summing over $\gamma$ and over the relevant range of $k$) $L^2$ estimates for the term $\sigma_k\ast f_q^{\gamma}-g_k^{m(k, q, \gamma), \gamma}$. This term represents the ``projection of $\sigma_k\ast f_q^{\gamma}$ onto low heights."

\subsection*{Discarding the heavy terms via exceptional sets}
Recall that the relevant range of $k$ in Proposition \ref{mainprop3} is $\mathcal{K}_2$, given by
\begin{align}\label{k2range}
\max(\log(l(q))(d-1), \log(\alpha^{-1}\gamma))\le k(d-1)\le\log(\alpha^{-1}\gamma)+100\Log^2(\alpha^{-1}).
\end{align} 
Now, note that $\sum_{i\ge k(d-1)-\log(\gamma\alpha^{-1})+\log(\log(\log(\alpha^{-1})))}(g_k^{i, \gamma}-g_k^{i+1, \gamma})$ is supported in the exceptional set $\bigcup_{q, k, \gamma}A_{q, k, \gamma}$. Indeed, for all such $\delta$-grains $\omega_l$ which appear in the expression defining $g_k^{i, \gamma}$ for $$i\ge k(d-1)-\log(\gamma\alpha^{-1})+\log(\log(\log(\alpha^{-1}))),$$ there is a ``heavy" rectangle containing $\omega_l$ such that $\supp(\sigma_{k, j, n}\ast \chi_R)$ is contained in the exceptional set $\bigcup_{q, k, \gamma}A_{q, k, \gamma}$, and this clearly implies that $\text{supp}(\sigma_{k, j, n}\ast (f_q^{\gamma}\chi_{\omega_l}))$ is contained in the exceptional set. Since the support of $\sigma_{k, \omega_l}^{i, \gamma}$ is comprised of the union of supports of such $\sigma_{k, j, n}$, it follows that $\sum_{i\ge k(d-1)-\log(\gamma\alpha^{-1})+\log(\log(\log(\alpha^{-1})))}(g_k^{i, \gamma}-g_k^{i+1, \gamma})$ is supported in the exceptional set. We summarize this as the following remark.
\begin{remark}\label{penmark}
Note that $\sum_{i\ge k(d-1)-\log(\gamma\alpha^{-1})+\log(\log(\log(\alpha^{-1})))}(g_k^{i, \gamma}-g_k^{i+1, \gamma})$ is supported in the exceptional set $\bigcup_{q, k, \gamma}A_{q, k, \gamma}$.
\end{remark}
\subsection*{Handling the intermediate terms via $L^1$ estimates}
If we temporarily ignore the first term $\sigma_k\ast f_q^{\gamma}-g_k^{m(k, q, \gamma), \gamma}$ (which we will deal with later using $L^2$ estimates) it then remains to consider 
\begin{align*}
\sum_{m(k, q, \gamma)\le i\le k(d-1)-\log(\gamma\alpha^{-1})+\log(\log(\log(\alpha^{-1})))}(g_k^{i, \gamma}-g_k^{i+1, \gamma}).
\end{align*}
We will in fact sum these terms in $L^1$ and prove the following estimate.
\begin{lemma}\label{1lem}
Let $\mathcal{K}_2$ be defined as in (\ref{k2range}). The $L^1$ norm of
\begin{align*}
\sum_{k\in\mathcal{K}_2}\bigg(\sum_{m(k, q, \gamma)\le i\le k(d-1)-\log(\gamma\alpha^{-1})+\log(\log(\log(\alpha^{-1})))}(g_k^{i, \gamma}-g_k^{i+1, \gamma})\bigg)
\end{align*}
is $\lesssim\Log^3(\alpha^{-1})\Norm{f_q^{\gamma}}_1$. 
\end{lemma}

\begin{proof}[Proof of Lemma \ref{1lem}]
It suffices to show that for any $j\ge 0$,
\begin{align*}
\sum_{k\in\mathcal{K}_2}(g_k^{m(k, q, \gamma)+j, \gamma}-g_k^{m(k, q, \gamma)+j+1, \gamma})
\end{align*}
has $L^1$ norm $\lesssim\Norm{f_q^{\gamma}}_1$. To see this, note that
\begin{multline*}
\sum_{k\in\mathcal{K}_2}(g_k^{m(k, q, \gamma)+j, \gamma}-g_k^{m(k, q, \gamma)+j+1, \gamma})=
\\
=\sum_{\delta-\text{cubes }\omega_l}\sum_{k\in\mathcal{K}_2}(\sigma_{k, \omega_l}^{m(k, q, \gamma)+j, \gamma}-\sigma_{k, \omega_l}^{m(k, q, \gamma)+j+1, \gamma})\ast f_q^{\gamma}\chi_{\omega_l}.
\end{multline*}
Reindexing the sum in $k$ with $k=i+\log(\gamma\alpha^{-1})/(d-1)$, we may bound this from above by
\begin{multline*}
\sum_{\delta\text{-cubes }\omega_l}
\sum_{i\ge 0}(\sigma_{\log(\gamma\alpha^{-1})/(d-1)+i;\, \omega_l}^{m(\log(\gamma\alpha^{-1})/(d-1)+i, q, \gamma)+j;\, \gamma}
\\
-\sigma_{\log(\gamma\alpha^{-1})/(d-1)+i;\,\omega_l}^{m(\log(\gamma\alpha^{-1})/(d-1)+i, q, \gamma)+j+1;\, \gamma})\ast{(f_q^{\gamma}\chi_{\omega_l})}.
\end{multline*}
Now note that for any fixed $\omega_l$, the angles subtended by the spherical caps supporting each term
\begin{align*}
\sigma_{\log(\gamma\alpha^{-1})/(d-1)+i;\, \omega_l}^{m(\log(\gamma\alpha^{-1})/(d-1)+i, q, \gamma)+j;\, \gamma}
-\sigma_{\log(\gamma\alpha^{-1})/(d-1)+i;\,\omega_l}^{m(\log(\gamma\alpha^{-1})/(d-1)+i, q, \gamma)+j+1;\, \gamma})
\end{align*}
in the $i$-sum are disjoint. Indeed, this follows directly from the definition of these measures, the telescoping nature of the decomposition, and the fact that $(d-1)\ge 1$ ensures that for different values of $k$, the difference of these measures live at different ``heights," and the differences of measures at consecutive heights isolates the height at which a certain angular piece first occurs. This disjointness implies that
\begin{align*}
\sum_{i\ge 0}\Norm{\bigg(\sigma_{\log(\gamma\alpha^{-1})/(d-1)+i;\, \omega_l}^{m(\log(\gamma\alpha^{-1})/(d-1)+i, q, \gamma)+j;\, \gamma}-\sigma_{\log(\gamma\alpha^{-1})/(d-1)+i;\,\omega_l}^{m(\log(\gamma\alpha^{-1})/(d-1)+i, q, \gamma)+j+1;\, \gamma}\bigg)\ast f_q^{\gamma}\chi_{\omega_l}}_1
\\
\lesssim \Norm{f_q^{\gamma}\chi_{\omega_l}}_1.
\end{align*}
Summing over all $\delta$-grains $\omega_l$ completes the proof of the lemma.
\end{proof}

\subsection*{Estimating the $L^2$ norm of the light term}
By Remark \ref{penmark} and Lemma \ref{1lem}, to prove Proposition \ref{mainprop3}, it suffices to prove
\begin{proposition}\label{mainprop4}
\begin{align}\label{rand}
|\{x:\bigg(\sum_q\sum_{k>l(q)}|\sum_{\gamma>\gamma_0:\,\log(\alpha^{-1}\gamma)\le k(d-1)\le \log(\alpha^{-1}\gamma)+100\log(\log(\alpha^{-1}))}\sigma_k\ast{f_q^{\gamma}}
\\
-g_k^{m(k, q, \gamma), \gamma}|^2\bigg)^{1/2}>\alpha\}|\lesssim\alpha^{-1}\Norm{f}_1.
\end{align}
\end{proposition}

\begin{remark}
Note that the right hand side of (\ref{rand}) is actually smaller by a factor of $\Log^3(\alpha^{-1})$ than what we actually need.
\end{remark}

\begin{proof}[Proof of Proposition \ref{mainprop4}]

Note that for each $k$, there are $\lesssim\log(\log(\alpha^{-1}))$ different values of $\gamma$ for which 
\begin{align*}
\log(\alpha^{-1}\gamma)\le k(d-1)\le\log(\alpha^{-1}\gamma)+100\log(\log(\alpha^{-1})).
\end{align*}
It follows that (\ref{rand}) reduces to showing that for a fixed $\gamma>\gamma_0$ and for $k\in\mathcal{K}_2$ we have
\begin{align*}
\Norm{\sigma_k\ast{f_q^{\gamma}}-g_k^{m(k, q, \gamma), \gamma}}_2^2\lesssim\alpha(\log(\log(\alpha^{-1})))^{-2}\Norm{f_q^{\gamma}}_1.
\end{align*}
The first step is to write
\begin{multline}\label{fstep}
\Norm{\sigma_k\ast f_q^{\gamma}-g_k^{m(k, q, \gamma), \gamma}}_2^2
=
\Norm{\sum_{\delta\text{-grains }\omega_l}(\sigma_k-\sigma_{k, \omega_l}^{m(k, q, \gamma), \gamma})\ast f_q^{\gamma}\chi_{\omega_l}}_{L^2}^2
\\
\lesssim\sum_{\delta\text{-grains }\omega_l}\left<(\sigma_k-\sigma_{k, \omega_l}^{m(k, q, \gamma), \gamma})\ast f_q^{\gamma}\chi_{\omega_l};\, \sigma_k\ast f_q^{\gamma}\right>
\\
=\sum_{\delta\text{-grains }\omega_l}\left<f_q^{\gamma}\chi_{\omega_l};\, \sigma_k\ast(\sigma_k-\sigma_{k, \omega_l}^{m(k, q, \gamma), \gamma})\ast f_q^{\gamma}\right>.
\\
\end{multline}

\subsubsection*{Domination of the kernel $\sigma_k\ast\sigma_k$ by linear combinations of characteristic functions of hyperrectangles}
Recall that we have the pointwise estimate
\begin{align*}
\sigma_k\ast\sigma_k(x)\lesssim 2^{-k(d-1)}|x|^{-1}\chi_{|x|\le 2^{k+1}}.
\end{align*}
Thus the kernel $\sigma_k\ast\sigma_k$ can essentially be decomposed as follows. Fix $q$ and $\gamma$, and let $\{c_i\}_{i=0}^N$ be the enumeration of the scales described earlier in (\ref{scale}). For $1\le i\le N$, let $\mathcal{R}_i$ be a collection of $(c_i/c_{i-1})^{d-1}$ many hyperrectangles of dimensions $c_i\times c_i\times\cdots\times c_i\times c_{i-1}$ centered at the origin, with short sides pointing in equally spaced directions, where $\{c_i\}_{i=1}^N$ are the scales described earlier. We may dominate\begin{multline}\label{domdec}
2^{-k(d-1)}|x|^{-1}\chi_{\gamma^{1/(d-1)}\le |x|\le \frac{1}{100}\max(l(q), (\gamma\alpha^{-1})^{1/(d-1)})}
\\
\lesssim\sum_{i=1}^N2^{-k(d-1)}c_i^{-(d-1)}c_{i-1}^{d-2}\sum_{R\in{\mathcal{R}_i}}\chi_{R}.
\end{multline}
Indeed, for each $i$, $c_i^{-(d-1)}c_{i-1}^{d-2}\sum_{r\in\mathcal{R}_i}\chi_R$ is essentially of size $c_i^{-1}$ for $|x|\approx c_i$, since for $|x|\approx c_i$ the hyperrectangles are essentially disjoint in the case that $d=2$, and for general $d$ there are 
$$\approx c_i^{-d}\times (c_i/c_{i-1})^{d-1}\times c_i^{d-1}c_{i-1}\approx (c_i/c_{i-1})^{d-2}$$ 
many hyperrectangles that intersect a given $x$. By similar reasoning, one sees that for $c_i\le |x|\le c_i$, we also have  $c_i^{-(d-1)}c_{i-1}^{d-2}\sum_{r\in\mathcal{R}_i}\chi_R$ is essentially of size $|x|^{-1}$.

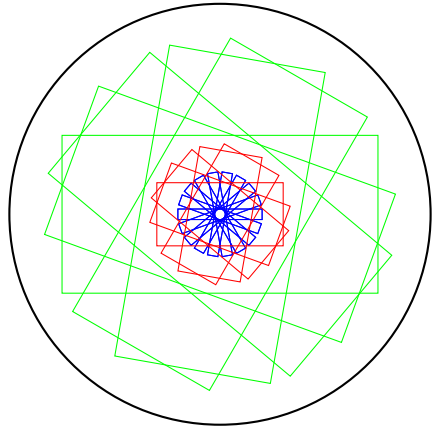
\begin{figure}
\begin{tikzpicture}[scale=0.7]

\draw[thick] (0,0) circle (4cm);

\draw[blue, rotate =20] (-0.8, -0.1) rectangle (0.8, 0.1);
\draw[blue, rotate =40] (-0.8, -0.1) rectangle (0.8, 0.1);
\draw[blue, rotate =60] (-0.8, -0.1) rectangle (0.8, 0.1);
\draw[blue, rotate =80] (-0.8, -0.1) rectangle (0.8, 0.1);
\draw[blue, rotate =100] (-0.8, -0.1) rectangle (0.8, 0.1);
\draw[blue, rotate =120] (-0.8, -0.1) rectangle (0.8, 0.1);
\draw[blue, rotate =140] (-0.8, -0.1) rectangle (0.8, 0.1);
\draw[blue, rotate =160] (-0.8, -0.1) rectangle (0.8, 0.1);
\draw[blue, rotate =180] (-0.8, -0.1) rectangle (0.8, 0.1);
\draw[blue, rotate =200] (-0.8, -0.1) rectangle (0.8, 0.1);
\draw[blue, rotate =220] (-0.8, -0.1) rectangle (0.8, 0.1);
\draw[blue, rotate =240] (-0.8, -0.1) rectangle (0.8, 0.1);
\draw[blue, rotate =260] (-0.8, -0.1) rectangle (0.8, 0.1);
\draw[blue, rotate =280] (-0.8, -0.1) rectangle (0.8, 0.1);
\draw[blue, rotate =300] (-0.8, -0.1) rectangle (0.8, 0.1);
\draw[blue, rotate =320] (-0.8, -0.1) rectangle (0.8, 0.1);
\draw[blue, rotate =340] (-0.8, -0.1) rectangle (0.8, 0.1);
\draw[blue, rotate =360] (-0.8, -0.1) rectangle (0.8, 0.1);

\draw[red, rotate =80] (-1.2, -0.6) rectangle (1.2, 0.6);
\draw[red, rotate =160] (-1.2, -0.6) rectangle (1.2, 0.6);
\draw[red, rotate =240] (-1.2, -0.6) rectangle (1.2, 0.6);
\draw[red, rotate =320] (-1.2, -0.6) rectangle (1.2, 0.6);
\draw[red, rotate =360] (-1.2, -0.6) rectangle (1.2, 0.6);

\draw[green, rotate =0] (-3, -1.5) rectangle (3, 1.5);
\draw[green, rotate =80] (-3, -1.5) rectangle (3, 1.5);
\draw[green, rotate =160] (-3, -1.5) rectangle (3, 1.5);
\draw[green, rotate =240] (-3, -1.5) rectangle (3, 1.5);
\draw[green, rotate =320] (-3, -1.5) rectangle (3, 1.5);

\end{tikzpicture}
\caption{Domination of the kernel $\sigma_k\ast\sigma_k$ in the sphere of radius $\frac{1}{100}\max(l(q), (\gamma\alpha^{-1})^{1/(d-1))}$ centered at the origin.}\label{fig2}
\end{figure}

\subsubsection*{Eliminating bad hyperrectangles}
Now fix some $i$ with $1\le i\le N$, and suppose that $R$ is a hyperrectangle in $\mathcal{R}_i$ such that 
\begin{align*}
\int_{\omega_l+R}|f_q^{\gamma}|\gtrsim c_{i-1}\gamma 2^{m(k, q, \gamma)}.
\end{align*}
Then by definition, the support of $\sigma_{k, \omega_l}^{m(k, q, \gamma), \gamma}$ contains a spherical cap of angular width $50c_{i-1}/c_i$ with some normal parallel to the short side of $R$. This implies that $\sigma_k\ast(\sigma_k-\sigma_{k, \omega_l}^{m(k, q, \gamma), \gamma})$ is supported outside of the set $(R)_1$, where we define $(R)_1$ to be $R\cap\{x:\,\frac{1}{10}c_i\le |x|\le 10 c_i\}$. 
\newline
\indent
Indeed, for any $x\in\mathbb{R}^d$ in the support of $\sigma_k\ast(\sigma_k-\sigma_{k, \omega_l}^{m(k, q, \gamma), \gamma})$ with $\frac{1}{10}c_i\le |x|\le 10 c_i$, we require there to exist $y$ on the sphere of radius $2^k$ centered at the origin such that $x-y$ is also on the sphere of radius $2^k$ centered at the origin, but outside the cap  of angular width $50c_{i-1}/c_i$ with some normal parallel to the short side of $R$. Suppose toward a contradiction that $x\in R\cap\{z:\,\frac{1}{10}c_i\le |z|\le 10 c_i\}$. But for any such $x-y$, we have that $(x-y)+(R)_1$ lies outside the sphere of radius $2^k$, since $R$ will be transverse to the boundary of the sphere at $x-y$ (see Figure $3$). Thus we have verified our claim that  $\sigma_k\ast(\sigma_k-\sigma_{k, \omega_l}^{m(k, q, \gamma), \gamma})$ is supported outside of the set $(R)_1$.

\begin{figure}
\begin{tikzpicture}[scale=0.7]

\draw[thick] (3,0) arc (25:325:4cm);
\draw[dashed] (-0.5,-1.8) -- (3,0);
\draw[dashed] (-0.5,-1.8) -- (-3.6,.4);

\draw[dashed] (-0.5,-1.8) -- (-4,-3.6);
\draw[dashed] (-0.5,-1.8) -- (2.6,-4);
\draw (.5,-2.5) arc (320:395:1cm) node[anchor=north west] {\small \,\,$50c_{i-1}/c_{i}$};
\filldraw (0, 2.3) circle (2pt) node[anchor=north east] {\small \,$x-y$};
\filldraw[cyan] (-.1, .6) rectangle (.2, 1.8);
\filldraw[cyan] (-.1, 2.8) rectangle (.2, 4);

\end{tikzpicture}
\caption{The two blue rectangles represent the set $x-y+(R)_1$.}\label{fig3}
\end{figure}
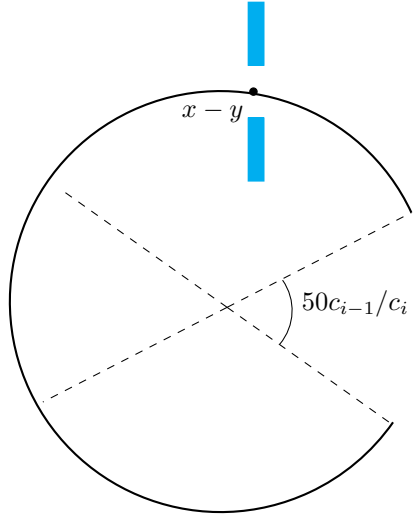
Repeating this process, if 
\begin{align*}
\int_{\omega_l+(R\setminus(R)_1)}|f_q^{\gamma}|\gtrsim c_{i-1}\gamma2^{m(k, q, \gamma)},
\end{align*}
then by definition, the support of $\sigma_{k, \omega_l}^{m(k, q, \gamma), \gamma}$ contains a spherical cap of angular width $50\cdot 2c_{i-1}/c_i$ with some normal parallel to the short side of $R$. As before, this implies that $\sigma_k\ast(\sigma_k-\sigma_{k, \omega_l}^{m(k, q, \gamma), \gamma})$ is supported outside of the set $(R)_2$, where we define $(R)_2$ to be $R\cap\{x:\,\frac{1}{20}c_i\le |x|\le 10c_i\}$. Repeating again, if
\begin{align*}
\int_{\omega_l+(R\setminus (R)_2)}|f_q^{\gamma}|\gtrsim c_{i-1}\gamma2^{m(k, q, \gamma)},
\end{align*}
then $\sigma_{k, \omega_l}^{m(k, q, \gamma), \gamma}$ contains a spherical cap of angular width $50\cdot 4c_{i-1}/c_i$ with some tangent parallel to the long side of $R$. This implies that $\sigma_k\ast(\sigma_k-\sigma_{k, \omega_l}^{m(k, q, \gamma), \gamma})$ is supported outside of the set $(R)_3$, where we define $(R)_3$ to be the set $R\cap\{x:\,\frac{1}{40}c_i\le |x|\le 10c_i\}$. We continue this process until stage $L$ when 
\begin{align*}
\int_{(\omega_l+(R\setminus (R)_L)}|f_q^{\gamma}|\lesssim c_{i-1}\gamma2^{m(k, q, \gamma)},
\end{align*}
and 
$\sigma_k\ast(\sigma_k-\sigma_{k, \omega_l}^{m(k, q, \gamma), \gamma})$ is supported outside of the set $(R)_L$. (Note that this must eventually happen if $\sigma_k-\sigma_{k, \omega_l}^{m(k, q, \gamma), \gamma}$ is not identically $0$, since the set $(R)_L$ can potentially increase by continuing this process up to $R\cap\{x: 10c_{i-1}\le |x|\le 10c_i\}$, which would imply that $\sigma_k-\sigma_{k, \omega_l}^{m(k, q, \gamma), \gamma}$ is identically $0$.) 
\newline
\indent
For convenience, we summarize the above in the following lemma.
\begin{lemma}\label{2endlemma}
Fix a $\delta$-grain $\omega_l$. For any hyperrectangle $R\in\mathcal{R}_i$, there is a subset $(R)_L\subset R$ such that
\begin{align*}
\int_{\omega_l+(R\setminus (R)_L)}|f_q^{\gamma}|\lesssim c_{i-1}\gamma 2^{m(k, q, \gamma)}
\end{align*} 
and $\sigma_k\ast(\sigma_k-\sigma_{k, \omega_l}^{m(k, q, \gamma), \gamma})$ is supported outside of the set $(R)_L$.
\end{lemma}

\subsubsection*{Finishing up the proof}
Lemma \ref{2endlemma} implies that for each $\delta$-grain $\omega_l$ and each hyperrectangle $R\in\mathcal{R}_i$, there is a function $h_R$ with $\int|h_R|\le c_{i-1}2^{m(k, q, \gamma)}\gamma$ so that by (\ref{fstep}) we may dominate
\begin{multline}\label{end1}
\Norm{\sigma_k\ast f_q^{\gamma}-g_k^{m(k, q, \gamma), \gamma}}_2^2=\sum_{{\delta}-\text{grains }\omega_l}\left<f_q^{\gamma}\chi_{\omega_l};\,\sigma_k\ast (\sigma_k-\sigma_{k, \omega_l}^{m(k, q, \gamma), \gamma})\ast f_q^{\gamma}\right>.
\\
\le
\sum_{{\delta}-\text{grains }\omega_l}\left<{f_q^{\gamma}\chi_{\omega_l}}, 2^{-k(d-1)}|x|^{-1}\chi_{(\gamma^{1/(d-1)}\le |x|\le\frac{1}{100}\max(l(q), (\gamma\alpha^{-1}))^{1/(d-1)})^c}\ast{f_q^{\gamma}}\right>
\\
+\sum_{{\delta}-\text{grains }\omega_l}\sum_{i=1}^N2^{-k(d-1)}c_i^{-(d-1)}c_{i-1}^{(d-2)}\sum_{R\in\mathcal{R}_i}\left<{f_q^{\gamma}\chi_{\omega_l}}, \chi_R\ast{h_R}\right>.
\end{multline}
It is not difficult to show that
\begin{multline*}
\left<{f_q^{\gamma}\chi_{\omega_l}}, 2^{-k(d-1)}|x|^{-1}\chi_{(\gamma^{1/(d-1)}\le |x|\le\frac{1}{100}\max(l(q), (\gamma\alpha^{-1})^{1/(d-1)})^c}\ast{f_q^{\gamma}}\right>
\\
\lesssim\alpha\Norm{f_q^{\gamma}\chi_{\omega_l}}_1.
\end{multline*}
Indeed, we have
\begin{multline}\label{end2}
\left<{f_q^{\gamma}\chi_{\omega_l}}, 2^{-k(d-1)}|x|^{-1}\chi_{(\gamma^{1/(d-1)}\le |x|\le\frac{1}{100}\max(l(q), (\gamma\alpha^{-1})^{1/(d-1)})^c}\ast{f_q^{\gamma}}\right>
\\
\lesssim 2^{-k(d-1)}\sum_{l:\,2^l\le \gamma^{1/(d-1)}}\int \int_{|x-y|\approx 2^l}\frac{1}{|x-y|}f_q^{\gamma}(y)\,dy\,f_q^{\gamma}(x)\chi_{\omega_l}(x)\,dx
\\
\sum_{l:\,\gamma^{1/(d-1)}<\frac{1}{100}\max(l(q), (\gamma\alpha^{-1})^{1/(d-1)})\le 2^l\le100\max(l(q), (\gamma\alpha^{-1})^{1/(d-1)})}\int \int_{|x-y|\approx 2^l}\frac{1}{|x-y|}
\\
\times f_q^{\gamma}(y)\,dy\,f_q^{\gamma}(x)\chi_{\omega_l}(x)\,dx\bigg)
\end{multline}
Using the fact that $|f_q^{\gamma}|\le 1$ for the first term and using (\ref{width3}) for the second term, we may bound this by
\begin{multline}
\left<{f_q^{\gamma}\chi_{\omega_l}}, 2^{-k(d-1)}|x|^{-1}\chi_{(\gamma^{1/(d-1)}\le |x|\le\frac{1}{100}\max(l(q), (\gamma\alpha^{-1})^{1/(d-1)})^c}\ast{f_q^{\gamma}}\right>
\\
\lesssim 2^{-k(d-1)}\Norm{f_q^{\gamma}\chi_{\omega_l}}_1\bigg(\sum_{l: 2^l\le \gamma^{1/(d-1)}}2^{l(d-1)}
\\
+\sum_{l: \gamma^{1/(d-1)}<\frac{1}{100}\max(l(q), (\gamma\alpha^{-1})^{1/(d-1)})\le 2^l\le100\max(l(q), (\gamma\alpha^{-1})^{1/(d-1)})}\gamma\bigg)
\\
\lesssim
 2^{-k(d-1)}\Norm{f_q^{\gamma}\chi_{\omega_l}}_1\gamma\lesssim (\gamma^{-1}\alpha)\Norm{f_q^{\gamma}\chi_{\omega_l}}_1\gamma\lesssim\alpha\Norm{f_q^{\gamma}\chi_{\omega_l}}_1.
\end{multline}
This gives a satisfactory bound for the first term occurring in the right hand side of (\ref{end1}). To bound the second term, we observe that since $\int|h_R|\le c_{i-1}2^{m(k, q, \gamma)}\gamma$, we have 
\begin{align}\label{end3}
\left<f_q^{\gamma}\chi_{\omega_l}, \chi_R\ast h_R\right>\lesssim c_{i-1}2^{m(k, q, \gamma)}\gamma\Norm{f_q^{\gamma}\chi_{\omega_l}}_1.
\end{align}
Combining (\ref{end1}), (\ref{end2}), and (\ref{end3}) and summing over all $i$ and all $\delta$-cubes $\omega_l$, using the fact that the cardinality of $\mathcal{R}_i$ is $\lesssim (c_i/c_{i-1})^{d-1}$ and $N\lesssim\log(\log(\alpha^{-1}))$, and recalling that 
\begin{align*}
m(k, q, \gamma):=k(d-1)-\log(\gamma\alpha^{-1})-\lceil{100\log(\log(\log(\alpha^{-1})))}\rceil,
\end{align*} 
we obtain
\begin{align*}
\Norm{\sigma_k\ast{f_q^{\gamma}}-f_{m(k, q, \gamma), k, \gamma}}_2^2\lesssim\alpha(\log(\log(\alpha^{-1})))^{-2}\Norm{f_q^{\gamma}}_1,
\end{align*}
which is the desired $L^2$ bound.
\end{proof}

\section{Proof of Proposition \ref{mainpropo2}}

We now give the proof of Proposition \ref{mainpropo2}, which relies on the simple observation that only $L^1$ and $L^2$ methods were used to prove Proposition \ref{midprop}. We reproduce the proposition here for convenience.
\begin{proposition}
For every $\epsilon>0$, there exists a constant $C(\epsilon)$ such that for all measurable $f$ and all $\alpha>0$ we have
\begin{multline}\label{mainineq}
|\{x\in\mathbb{R}^d:\,\mathcal{M}f(x)>\alpha\}|
\\
\le C(\epsilon)\alpha^{-1}\int|f(x)|\Log^3(|f(x)|/\alpha)(\Log^4(|f(x)|/\alpha))^{1+\epsilon}\,dx.
\end{multline}
\end{proposition}

\begin{proof}[Proof of Proposition \ref{mainpropo2}]
Without loss of generality assume that $f\le 1$. Clearly, we may also assume $\alpha\le f\le 1$ on the support of $f$. Thus on the support of $f$ there are at most $O(\Log^4(\alpha^{-1}))$ many dyadic level sets of $\Log^3(|f(x)|/\alpha)$, which we enumerate in decreasing order as $\{2^{-i}\}_{i=-100}^N$ for some $N=O(\Log^4(\alpha^{-1}))$. We write 
$$f=\sum_{i=1}^Nf_i,$$ 
where $f_i=f\chi_{x:\,\Log^3(|f(x)|/\alpha)\approx 2^i}$. 
\newline
\indent
Referring to the proof of Proposition \ref{midprop}, we now proceed with the reductions made in the beginning of the paper. Again, it suffices to prove the desired inequality with $f$ replaced by $\sum_if_i\chi_{\Omega_i^c}$, where $$\Omega_i:=\{x: M_{HL}(f_i)(x)>\alpha\}.$$ Let $\mathfrak{O}_i$ be the set of Whitney cubes at height $\alpha$ for $f_i$. As previously, by the $L^2$-boundedness of the lacunary spherical maximal operator and Cauchy-Schwarz we have
\begin{multline*}
\Norm{\sup_k|\sum_{i=1}^N\sum_{q\in\mathfrak{Q}_i}\Pi_q[f_{i, q}]\ast\sigma_k|}_2^2\lesssim\Norm{\sum_{i=1}^N\sum_{q\in\mathfrak{Q}_i}\Pi_q[f_{i, q}]}_2^2\lesssim_{\epsilon}\sum_{i=1}^Ni^{1+\epsilon}\Norm{\sum_{q\in\mathfrak{Q}_i}\Pi_q[f_{i, q}]}_2^2
\\
\lesssim_{\epsilon} \alpha\sum_{i=1}^Ni^{(1+\epsilon)}\Norm{\sum_{q\in\mathfrak{Q}_i}\Pi_q[f_{i, q}]}_1
\\
\lesssim_{\epsilon} \alpha\int|f(x)|(\Log^4(|f(x)|/\alpha))^{1+\epsilon}\,dx.
\end{multline*}
It thus remains to prove that
\begin{multline}\label{y1}
|\{x: \bigg(\sum_k|\sum_i\sum_{q_i\in\mathfrak{Q}_i}b_{q_i}\ast\sigma_k|^2\bigg)^{1/2}>\alpha\}|\lesssim_{\epsilon}\alpha^{-1}\int|f(x)|(\Log^3(|f(x)|/\alpha))
\\
\cdot(\Log^4(|f(x)|/\alpha))^{1+\epsilon}\,dx.
\end{multline}
By Cauchy-Schwarz, this further reduces to proving
\begin{multline}\label{y2}
|\{x: \bigg(\sum_k\sum_ii^{1+\epsilon}|\sum_{q_i\in\mathfrak{Q}_i}b_{q_i}\ast\sigma_k|^2\bigg)^{1/2}>\alpha\}|\lesssim_{\epsilon}
\\
\alpha^{-1}\int|f(x)|\Log^3(|f(x)|/\alpha)(\Log^4(|f(x)|/\alpha))^{1+\epsilon}\,dx.
\end{multline}

As previously, we may exploit the smoothness of the kernel $\sigma_k\ast\sigma_k$ and the cancellation of the atoms as in \cite{stw} to obtain
\begin{multline*}
\Norm{\bigg(\sum_{k}\sum_ii^{1+\epsilon}\sum_{q_i\ne q_i'}(b_{q_i}\ast\sigma_k)(\overline{b_{q_i'}\ast\sigma_k})\bigg)^{1/2}}_2^2
\\
\lesssim\sum_k\sum_i i^{1+\epsilon}\sum_{q_i\ne q_i'}\bigg|\left<b_{q_i}\ast\sigma_k, b_{q_i'}\ast\sigma_k\right>\bigg|
\\
\lesssim_{\epsilon}\alpha\int|f(x)|(\Log^4(|f(x)|/\alpha))^{1+\epsilon}\,dx.
\end{multline*}
We again throw away as exceptional sets the unions of the supports of $\sigma_k\ast q_i$ for $k$ such that $2^k\le l(q_i)$. We also have
\begin{multline*}
\sum_i\sum_{q_i}\sum_{k:\,2^k>l(q_i)}i^{(1+\epsilon)}\Norm{\Pi_{q_i}[f_{q_i}]\ast\sigma_k}_2^2\lesssim_{\epsilon}\sum_i i^{1+\epsilon}\sum_{q_i}\sum_{k:\,2^k>l(q_i)}\Norm{\alpha\chi_{q_i}\ast\sigma_k}_2^2
\\
\lesssim_{\epsilon}\alpha\sum_ii^{1+\epsilon}\Norm{f_i}_1
\lesssim_{\epsilon}\alpha\int|f(x)|(\Log^4(|f(x)|/\alpha))^{1+\epsilon}\,dx.
\end{multline*}
Thus (\ref{y2}) reduces to proving
\begin{multline}\label{y3}
|\{x: \bigg(\sum_i i^{1+\epsilon}\sum_{q_i\in\mathfrak{Q}_i}\sum_{k:\,2^k>l(q_i)}|\sigma_k\ast f_{q_i}|^2\bigg)^{1/2}>\alpha\}|
\\\lesssim_{\epsilon}\alpha^{-1}\int|f(x)|\Log^3(|f(x)|/\alpha)(\Log^4(|f(x)|/\alpha))^{1+\epsilon}\,dx.
\end{multline}
Now, we keep in mind that the remainder of the proof of Proposition \ref{mainpropo} uses only $L^1$ and $L^2$ methods. In fact, the proof shows that for each $i$, we can decompose
\begin{align*}
\sigma_{k\ast f_{q_i}}=g_{k, i, q_i, 1}+g_{k, i, q_i, 2}+g_{k, i, q_i, 3}, 
\end{align*}
where for each $i$ we have
\begin{align}\label{l2}
\sum_{q_i\in\mathfrak{Q}_i}\sum_{k:\,2^k>l(q_i)}\Norm{g_{k, i, q_i, 1}}_2^2\lesssim\alpha\int|f_i|\Log^3(|f_i(x)|/\alpha)\,dx,
\end{align}

\begin{align}\label{l1}
\sum_{q_i\in\mathfrak{Q}_i}\sum_{k:\,2^k>l(q_i)}\Norm{g_{k, i, q_i, 2}}_1\lesssim\int|f_i|\Log^3(|f_i(x)|/\alpha)\,dx,
\end{align}
and
\begin{align}\label{supp}
\bigcup_{i, k}\text{supp}(g_{k, i, q_i, 3})\lesssim\alpha^{-1}\int|f|.
\end{align}
To deal with (\ref{l2}), we multiply through by $i^{1+\epsilon}$ and sum in $i$ to obtain
\begin{multline}\label{e2}
\sum_i i^{1+\epsilon}\sum_{q_i\in\mathfrak{Q}_i}\sum_{k:\,2^k>l(q_i)}\Norm{g_{k, i, q_i, 1}}_2^2\lesssim_{\epsilon}\alpha\int|f_i|\Log^3(|f_i(x)|/\alpha)
\\
\cdot(\Log^4(|f_i(x)|/\alpha))^{1+\epsilon}\,dx.
\end{multline}
Summing (\ref{l1}) and (\ref{supp}) in $i$ and combining the resulting inequalities with (\ref{e2}) along with Chebyshev leads to the desired inequality (\ref{y3}).
 
\end{proof}

\end{document}